\documentclass[11pt]{amsart}
\usepackage{amssymb,amsmath,epsfig,mathrsfs, enumerate, xparse}
\usepackage[nodisplayskipstretch]{setspace}
\usepackage{graphics}
\usepackage[normalem]{ulem}
\usepackage{fancyhdr}
\usepackage{tikz}
\usepackage{amsthm}
\usepackage{esint}
\usepackage[margin=2.5cm]{geometry}
\usepackage{hyperref}
\hypersetup{
	colorlinks   = true,
	urlcolor     = blue,
	linkcolor    = blue,
	citecolor   = red ,
	bookmarksopen=true
}

\theoremstyle{plain}
\newtheorem{thrm}{Theorem}[section]
\newtheorem{lemma}[thrm]{Lemma}

\newtheorem{rmrk}[thrm]{Remark}

\allowdisplaybreaks
\begin{document}
	\newcommand{\psr}{P^-}
	\newcommand{\qr}{Q_\rho}
	\newcommand{\sn}{\mathbb{S}^{n-1}}
	\newcommand{\SL}{\mathcal L^{1,p}( D)}
	\newcommand{\Lp}{L^p( Dega)}
	\newcommand{\CO}{C^\infty_0( \Omega)}
	\newcommand{\Rn}{\mathbb R^n}
	\newcommand{\Rm}{\mathbb R^m}
	\newcommand{\R}{\mathbb R}
	\newcommand{\Om}{\Omega}
	\newcommand{\Hn}{\mathbb H^n}
	\newcommand{\aB}{\alpha B}
	\newcommand{\eps}{\ve}
	\newcommand{\BVX}{BV_X(\Omega)}
	\newcommand{\p}{\partial}
	\newcommand{\IO}{\int_\Omega}
	\newcommand{\bG}{\boldsymbol{G}}
	\newcommand{\bg}{\mathfrak g}
	\newcommand{\bz}{\mathfrak z}
	\newcommand{\bv}{\mathfrak v}
	\newcommand{\Bux}{\mbox{Box}}
	\newcommand{\e}{\ve}
	\newcommand{\X}{\mathcal X}
	\newcommand{\Y}{\mathcal Y}
	\newcommand{\W}{\mathcal W}
	\newcommand{\la}{\lambda}
	\newcommand{\vf}{\varphi}
	\newcommand{\rhh}{|\nabla_H \rho|}
	\newcommand{\Ba}{\mathcal{B}_\beta}
	\newcommand{\Za}{Z_\beta}
	\newcommand{\ra}{\rho_\beta}
	\newcommand{\na}{\nabla_\beta}
	\newcommand{\vt}{\vartheta}

	\numberwithin{equation}{section}

	\newcommand{\RN} {\mathbb{R}^N}
	\newcommand{\Sob}{S^{1,p}(\Omega)}
	\newcommand{\Dxk}{\frac{\partial}{\partial x_k}}
	\newcommand{\Co}{C^\infty_0(\Omega)}
	\newcommand{\Je}{J_\ve}
	\newcommand{\beq}{\begin{equation}}
		\newcommand{\bea}[1]{\begin{array}{#1} }
			\newcommand{\eeq}{ \end{equation}}
		\newcommand{\ea}{ \end{array}}
	\newcommand{\eh}{\ve h}
	\newcommand{\Dxi}{\frac{\partial}{\partial x_{i}}}
	\newcommand{\Dyi}{\frac{\partial}{\partial y_{i}}}
	\newcommand{\Dt}{\frac{\partial}{\partial t}}
	\newcommand{\aBa}{(\alpha+1)B}
	\newcommand{\GF}{\psi^{1+\frac{1}{2\alpha}}}
	\newcommand{\GS}{\psi^{\frac12}}
	\newcommand{\HFF}{\frac{\psi}{\rho}}
	\newcommand{\HSS}{\frac{\psi}{\rho}}
	\newcommand{\HFS}{\rho\psi^{\frac12-\frac{1}{2\alpha}}}
	\newcommand{\HSF}{\frac{\psi^{\frac32+\frac{1}{2\alpha}}}{\rho}}
	\newcommand{\AF}{\rho}
	\newcommand{\AR}{\rho{\psi}^{\frac{1}{2}+\frac{1}{2\alpha}}}
	\newcommand{\PF}{\alpha\frac{\psi}{|x|}}
	\newcommand{\PS}{\alpha\frac{\psi}{\rho}}
	\newcommand{\ds}{\displaystyle}
	\newcommand{\Zt}{{\mathcal Z}^{t}}
	\newcommand{\XPSI}{2\alpha\psi \begin{pmatrix} \frac{x}{\left< x \right>^2}\\ 0 \end{pmatrix} - 2\alpha\frac{{\psi}^2}{\rho^2}\begin{pmatrix} x \\ (\alpha +1)|x|^{-\alpha}y \end{pmatrix}}
	\newcommand{\ZZ}{ \begin{pmatrix} xx^{t} & (\alpha + 1)|x|^{-\alpha}x y^{t}\\
			(\alpha + 1)|x|^{-\alpha}x^{t} y &   (\alpha + 1)^2  |x|^{-2\alpha}yy^{t}\end{pmatrix}}
	\newcommand{\norm}[1]{\lVert#1 \rVert}
	\newcommand{\ve}{\varepsilon}
	\newcommand{\D}{\operatorname{div}}
	\newcommand{\avint}{\mathop{\ooalign{$\int$\cr$-$}}}

	\title[Gradient Continuity etc.]{Borderline gradient continuity for fractional heat type operators}

	\author{Vedansh Arya*}
	\address{Tata Institute of Fundamental Research\\
		Centre For Applicable Mathematics \\ Bangalore-560065, India}\email[Vedansh Arya]{vedansh@tifrbng.res.in}
	
	\author{Dharmendra Kumar**}
	\address{Tata Institute of Fundamental Research\\
	Centre For Applicable Mathematics \\ Bangalore-560065, India}\email[Dharmendra Kumar]{dharmendra2020@tifrbng.res.in}

	\thanks{*vedansh@tifrbng.res.in,**dharmendra2020@tifrbng.res.in}
	\thanks{Tata Institute of Fundamental Research, Centre For Applicable Mathematics, Bangalore-560065, India}



	%
	%
	%
	\keywords{Extension problem, Fractional heat operator, Master equations, Borderline gradient continuity}
	\subjclass{35R11, 35B65, 35K65, 35R09}
	
	\maketitle
	
	\begin{abstract}
	
	In this paper, we establish gradient continuity for solutions to \[ (\partial_t - \operatorname{div}(A(x) \nabla u))^s =f,\   s \in (1/2, 1), \] when $f$ belongs to the scaling critical function space $L(\frac{n+2}{2s-1}, 1)$. Our main results Theorems \ref{mthr} and  \ref{main1} can be seen  as a nonlocal  generalization of a well known result of Stein in the context of fractional heat type operators and sharpens some of the previous gradient continuity results which deals with $f$ in subcritical spaces. Our proof is based on an appropriate adaptation of  compactness arguments, which has its roots in a fundamental work of Caffarelli in \cite{Ca}.  
	
	\end{abstract}

	\tableofcontents
	
	\section{Introduction and the statement of the main result}
	In this article we prove  gradient continuity for the following nonlocal   operators $(\partial_t - \operatorname{div}( A(x) \nabla))^s$ which are modelled on the fractional heat operator $(\partial_t - \Delta)^s$  with critical scalar perturbations.  To provide some context to our work,  we note that the  study of the fractional heat operator $(\p_t - \Delta)^s$ was first proposed in M. Riesz' visionary papers \cite{R1} and \cite{R2}. This  nonlocal operator represents a basic model of the \emph{continuous time random walk} (CTRW) introduced by Montroll and Weiss in \cite{MW}. We recall that a CTRW is a generalization of a random walk where the wandering particle waits for a random time between jumps. It is a stochastic jump process with arbitrary distributions of jump lengths and waiting times. In \cite{MK}, Klafter and Metzler describe such processes by means of the nonlocal equation both in space and time
\[
\eta(t,x) = \int_0^\infty \int_{\mathbb{R}} \Psi(\tau,z) \eta(t-\tau,x-z) dz d\tau,
\]
which is an example of a \emph{master equation}. Such equations were introduced in 1973 by  Kenkre, Montroll and Shlesinger in \cite{KMS}, and they are presently receiving increasing attention by mathematicians also thanks to the work \cite{CSstein} of Caffarelli and Silvestre in which the authors establish the H\"older continuity of viscosity solutions of generalized master equations
\begin{equation}\label{masterL}
Lu(t,x) = \int_0^\infty\int_{\R^n} K(t,x;\tau,z) [u(t,x) - u(t-\tau,x-z)] dz d\tau = 0.
\end{equation}
On the kernel $K$ they assume that there exist $0<s<1$ and $\beta>2s$ such that for $0<c_1<c_2$  one can find $0<\la \le \Lambda$ for which for a.e. $(t,x)\in \mathbb{R}\times \mathbb{R}^n$ one has
\begin{equation}\label{master}
\begin{cases}
K(t,x;\tau,z) \ge \frac{\la}{|z|^{n+2s+\beta}}\ \ \ \ \ \text{when}\ c_1 |z|^\beta \le \tau \le c_2 |z|^\beta,
\\
K(t,x;\tau,z) \le \frac{\Lambda}{|z|^{n+2s+\beta} + \tau^{n/\beta + 1 + 2s/\beta}}.
\end{cases}
\end{equation}
The pseudo-differential operator $(\p_t - \Delta)^s$  which is defined  via the Bochner's  subordination principle  in the following way ( see for instance \cite{Ba}, \cite{ST}, \cite{BG} ),
\begin{align}\label{hs}
(\p_t - \Delta)^s u(t,x) & =  \frac{s}{\Gamma(1-s)}  \int_0^\infty  \int_{\Rn} \tau^{-s-1} G(\tau,z) [u(t,x) - u(t-\tau,x-z)] dz d\tau, 
\\
& = \frac{s}{\Gamma(1-s)}  \int_{-\infty}^t  \int_{\Rn} (t-\tau)^{-s-1} G(t-\tau,x-z) [u(t,x) - u(\tau,z)] dz d\tau,
\notag
\end{align}
where $G(\tau,z) = (4\pi \tau)^{-\frac n2} e^{-\frac{|z|^2}{4\tau}}$ is the standard heat kernel and $\Gamma(z)$ indicates Euler gamma function, is seen to be a special case of the master equation. We mention that  recently  Nystr\"om and Sande \cite{NS} and Stinga and Torrea \cite{ST} have independently adapted  to the fractional heat operator the celebrated extension procedure of Caffarelli and Silvestre in \cite{CS} which can be described as follows.

Given $s\in (0,1)$  we introduce the parameter 
\[
a = 1-2s\in (-1,1),
\]
 and indicate with $U = U(t,X)$ the solution to the following \emph{extension problem}  
 \begin{equation}\label{epb}
\begin{cases}
y^a \frac{\p U}{\p t} = \operatorname{div}_{X}(y^a \nabla_{X} U),\ \ \ \ \ \ \ (t,X)\in \mathbb{R} \times \mathbb{R}^n\times (0, \infty) ,
\\
U(t,x,0) = u(t,x),\ \ \ \ \ \ \ \ \ \ \ \ \ (t,x)\in \mathbb{R} \times \mathbb{R}^n.
\end{cases}
\end{equation}

Using an appropriate Poisson representation  of the extension problem it was proved in \cite{NS} and \cite{ST}, see also Section 3 in \cite{BG} for details, that one has
in $L^2(\mathbb{R} \times \mathbb{R}^n)$
\begin{equation}\label{dtn2}
- \frac{2^{2s-1} \Gamma(s)}{\Gamma(1-s)} \underset{y\to 0^+}{\lim} y^a \frac{\partial U}{\partial y}(t,x,y) = (\partial_t - \Delta)^s u(t,x).
\end{equation}

Such an extension problem has been generalized for fractional powers of  variable coefficient operators such  as $(\partial_t - \operatorname{div} ((A(x) \nabla))^s$ in  \cite{BDS} and  \cite{BS}. We also refer to \cite{BGMN}  for a   generalization   of the extension problem in the subelliptic situation.  We would like to mention  that the study of the fractional heat type operators as well as the related  extension problem has received a lot of attention in recent times, see  for instance \cite{Au}, \cite{ACM}, \cite{AT}, \cite{BG}, \cite{BDGP1}, \cite{BDGP2}, \cite{BS}, \cite{DP}, \cite{LN}, \cite{LLR}. We would also like to mention that extention problem is prototype of the equations with general $A_2$ weight studied by 
Chiarenza and Serapioni in \cite{CSe}.

 Now we will state our main result: Consider the following problem
		\begin{equation}\label{mainpr1}
			\begin{cases}
				y^a \partial_t U - \text{div}(y^a B(x) \nabla U)=0&\hbox{in}~Q^*_1 \\
				-y^a U_y\big|_{y=0}=f&\hbox{on}~Q_1,
			\end{cases}
		\end{equation}
	where $a=1-2s;$ for $1/2<s<1$, $$ B= \begin{bmatrix}
		A & 0\\
		0 & 1
	\end{bmatrix} $$
	and $A$ is a uniformly elliptic matrix with Dini modulus of continuity $\omega_A$ and $f$ belongs to the scaling critical function  space   $L\big(\frac{n+2}{2s-1},1\big).$
	\begin{thrm}\label{mthr}
		Let $U$ be a solution of (\ref{mainpr1}) in $Q_1^*$. Then there exist modulus of continuity $K$ depending on ellipticity, $\omega_A,$ $n,$ $s$ and $f$ such that for all $(t_1,X_1),(t_2,X_2) \in Q_{1/2}^*$
	\begin{align}
		|\nabla U(t_1,X_1)-\nabla U(t_2,X_2)| \le C K(|(t_1,X_1)-(t_2,X_2)|)
	\end{align}
and \begin{align*}
	| U(t_1,X)- U(t_2,X)| \le C K(\sqrt{|t_1-t_2}|)\sqrt{|t_1-t_2|},
\end{align*}
where $C^2=\int_{Q_1}U(t,x,0)^2 dtdx + \int_{Q_1^*}U(t,X)^2 dtdX.$
	\end{thrm}
	
	In view of the extension problem for $(\partial_t - \operatorname{div}(B(x) \nabla ))^s$  in \cite{BS}, we obtain consequently that the following regularity result holds for the nonlocal fractional heat type problem. 
	
	\begin{thrm}\label{main1}
	Let $s \in (1/2, 1)$ and let  $u$ solve $(\partial_t - \operatorname{div}(A(x) \nabla u))^s =f$ where $A(x)$ is uniformly elliptic with Dini coefficients and $f \in L\big(\frac{n+2}{2s-1},1\big).$ Then $\nabla_x u$ is continuous.	\end{thrm}

	Now to put our results in the right historical perspective, we note that in 1981,  E. Stein in his work \cite{MR607898} showed the following ``limiting" case of Sobolev embedding theorem.
\begin{thrm}\label{stein}
Let $L(n,1)$ denote the standard Lorentz space, then the following implication holds:
\[\nabla v \in L(n,1) \ \implies \ v\  \text{is continuous}.\]  
\end{thrm}
The Lorentz space $L(n,1)$ appearing in Theorem \ref{stein}  consists of those measurable functions $g$ satisfying the condition
\[
\int_{0}^{\infty} |\{x: g(x) > t\}|^{1/n} dt < \infty.
\]
Theorem \ref{stein} can be regarded as the limiting case of Sobolev-Morrey embedding that asserts
\[
\nabla v \in L^{n+\ve} \implies v \in C^{0, \frac{\ve}{n+\ve}}.
\]
Note that indeed $L^{n+\ve} \subset L(n, 1) \subset L^{n}$ for any $\ve>0$ with all the inclusions being strict.  Now Theorem \ref{stein} coupled with the standard Calderon-Zygmund theory  has the following interesting consequence.
\begin{thrm}\label{st} If $\Delta u \in L(n,1)$ then this implies $\nabla u$ is continuous.
\end{thrm} 
Similar result holds  in the parabolic situation for more general variable coefficient operators   when $f \in L(n+2, 1)$. The analogue of Theorem \ref{st}  for general nonlinear and possibly degenerate elliptic and parabolic equations has become accessible not so long ago through a rather sophisticated and powerful nonlinear potential theory (see for instance \cite{MR2823872,MR2900466,MR3174278} and the references therein).  The first breakthrough in this direction came up in the work  of Kuusi and Mingione in \cite{MR3004772} where they showed that the analogue of Theorem \ref{st}  holds for operators modelled after the variational $p$-laplacian. Such a result  was subsequently generalized to $p$-laplacian type systems by the same authors in  \cite{MR3247381}.

Since then, there has been several generalizations of Theorem \ref{st} to operators with other kinds of nonlinearities and in context of  fully nonlinear elliptic equations, the  analogue of Theorem \ref{st} has been established by Daskalopoulos-Kuusi-Mingione in  
\cite{DKM}. We also refer to \cite{Ak} for the corresponding boundary regularity result and also  to \cite{BM} for a similar result in the context of the game theoretic normalized $p-$Laplacian operator. Our main result Theorem \ref{main1} can thus of thought of as  a nonlocal generalization of the Stein's theorem in the sense that $s \to 1$, it exactly reproduces the result Theorem \ref{stein} for time independent $f$. Moreover, Theorem \ref{main1} is also seen to be the limiting case of Theorem 1.2 ii) in \cite{BS} which instead deals gradient H\"older continuity  in case of subcritical scalar perturbations when $f \in L^{p}$ for $p > \frac{n+2}{2s-1}$. As the reader will see,  our proof is based on a somewhat delicate adaptation of the Caffarelli style compactness arguments in \cite{Ak}    to the degenerate Neumann problem in \eqref{mainpr1}. 

The paper is organized as follows. In section \ref{s:n}, we introduce some basic notations and gather some preliminary results.  In section \ref{s:m}, we prove our main result. Finally in the appendix, we give a self contained proof of a basic existence result that is used in our approximation Lemma \ref{closeness}.

	\medskip

	\textbf{Acknowledgment:}
	We are thankful to Agnid Banerjee for various helpful discussions and suggestions.

	\section{Notations and Preliminaries}\label{s:n} We will denote generic point of thick space $\R \times \R^n \times \R$ by $(t,X)=(t,x,y).$ In general we will identify thin space $\R \times \R^n \times \{0\}$ by $\R \times \R^n$ and its generic point will be denoted by $(t,x).$ We will define cubes and balls in thin and thick space as following:
	\begin{align*}
		\qr(x_0,t_0) &=(t_0-\rho^2,t_0+\rho^2)\times \{x \in \R^n:||x-x_0|| < \rho \};\\
		\qr^*(x_0,t_0,y_0) &=(t_0-\rho^2,t_0+\rho^2)\times \{x \in \R^n:||x-x_0|| < \rho \}\times (y_0,y_0+\rho);\\
		\textbf{Q}_{\rho}(x_0,t_0,y_0)&=(t_0-\rho^2,t_0+\rho^2)\times \{x \in \R^n:||x-x_0|| < \rho \}\times (y_0-\rho,y_0+\rho).
	\end{align*}
 Similary we define $B_r$, $B^*_r$ and $\textbf{B}_r.$ For $(t_1,X_1)$ and $(t_2,X_2)$, $$|(t_1,X_1)-(t_2,X_2)|:=\text{max}\{\sqrt{|t_1-t_2|},||X_1-X_2||\},$$ where $||X_1-X_2||$ denotes Euclidean norm. We denote $Q_{\rho}(0,0)$ by $Q_{\rho}$ etc.. 
	For $((t_1,t_2)\times \Omega) \subset \R \times \R^{n+1}$, 
	$\partial_p ((t_1,t_2)\times \Omega)$ will denote the parabolic boundary of  $(t_1,t_2)\times \Omega$  and defined as:
	\begin{align*}
		\partial_p((t_1,t_2)\times \Omega)&=(t_1,t_2)\times \Omega\cup[t_1,t_2]\times\partial\Omega.
	\end{align*}
	We will denote $L^2(Q_1^*)$ with measure $y^adtdX$ by $L^2_a(Q_1^*).$ Similary we define $H^1_a(Q_1^*)$ or $H^1_{0,a}(Q_1^*).$ Sometimes we will denote   $C([-1,1];L^2_a(B_1^*))$ by $C(-1,1;L^2_a(B_1^*))$ and other spaces in similar fashion.
	 Let $F=F(t,x)=(F_1,\ldots,F_n,F_{n+1})$ be an $\R^{n+1}$-valued vector field defined on $Q^*_1$
	such that $$F_{n+1} = 0, \hspace{2mm}|F|\in L^2(Q_1^*)\hspace{2mm}\text{and} \hspace{2mm}f=f(t,x)\in L^2(Q_1).$$
	
	Now consider the following local problem
	\begin{equation}\label{mainpr}
		\begin{cases}
			y^a \partial_t U - \text{div}(y^a B(x) \nabla U)=- \text{div}(y^a F)&\hbox{in}~Q^*_1 \\
			-y^a U_y\big|_{y=0}=f&\hbox{on}~Q_1.
		\end{cases}
	\end{equation}

	We say that $U \in C([-1,1];L^2_a(B_1^*)) \cap L^2([-1,1];H^1_a(B_1^*))$ is a weak solution to (\ref{mainpr}) if for every $-1<t_1<t_2<1$ 
	\begin{align}\label{weaksoldef}
		\int_{B^*_1}y^a U \phi|_{t=t_1}^{t=t_2}dX-\int_{t_1}^{t_2}\int_{B^*_1}y^a U \partial_t \phi dt dX +\int_{t_1}^{t_2}\int_{B^*_1}y^a (B(x) \nabla U) \cdot \nabla \phi dt dX\\
		=\int_{t_1}^{t_2}\int_{B_1}f(t,x)\phi(t,x,0)dt dX +\int_{t_1}^{t_2}\int_{B^*_1} y^a F \cdot \nabla \phi dt dX
	\end{align}
	holds for all $\phi \in H^1([-1,1];L^2_a(B_1^*)) \cap L^2([-1,1];H^1_a(B_1^*))$ such that $\phi=0$ on $\partial_pQ_1^*\setminus (Q_1 \times \{0\}).$ Before proceeding further, we make the following discursive remark.
	
	\begin{rmrk} We say a constant to be universal if it depends only on ellipticity, $n,$ $s,$ $L\big(\frac{n+2}{2s-1},1\big)$ norm of $f$ and the fact that $\int_{0}^{1}\frac{\omega_A(s)}{s}ds < \infty.$
	\end{rmrk} 	
	
	We will need the following lemmas from \cite{BS}.
	\begin{lemma}\label{energy}
		Assume that $U$ is a weak solution to \eqref{mainpr} with $F$ as described above.
		Then, for each $\varphi \in C^{\infty}_c(Q_1\times[0,1))$ and for each  $t_1,t_2  \in (-1,1)$ with $t_1 < t_2$, 
		\begin{align*}
			\sup_{t_1<t<t_2}\int_{B^*_1}y^aU^2 \varphi^2 \,dX + \int^{t_2}_{t_1}\int_{B^*_1} &y^a \varphi^2 |\nabla U|^2\,dXdt \\
			& \leq C\bigg[\int^{t_2}_{t_1}\int_{B^*_1} y^a \left(( |\partial_t (\varphi^2)|+|\nabla \varphi|^2)U^2 + |F|^2 \varphi^2 \right)\,dXdt  \\
			&\qquad+ \int^{t_2}_{t_1}\int_{B_1} (\varphi(t,x,0))^2 |U(t,x,0)||f(t,x)|\,dxdt  \bigg] \\
			&\qquad+ \int_{B^*_1}y^aU^2(t_1,X) \varphi^2(t_1,X) \,dX,
		\end{align*}
		where $C = C\big(n, s, \mathrm{ellipticity}\big)>0$.
	\end{lemma}

	\begin{lemma}\label{reg}
		Let $W$ be a weak solution to 
		\begin{equation}\label{16JULY4}
			\left\{
			\begin{aligned}{}
				y^a \partial_tW   &= \operatorname{div}(y^a \nabla W) &\text{in} \hspace{2mm}Q^*_{1} \\ 
				-y^a W_y\big|_{y=0} &= 0 &\text{on} \hspace{2mm}Q_{1}.
			\end{aligned}
			\right.
		\end{equation}
		Then we have the following estimates:
		\begin{enumerate}[$(i)$]
			\item For $k \in \mathbb{N} \cup \{0\}$, multi-index $\alpha$ and  $Q_r(t_0,x_0)\subset Q_1,$ we have
			$$\sup_{Q_{\frac{r}{2}}(t_0,x_0) \times [0, \frac{r}{2}) } \Big|\partial^k_t D^{\alpha}_x W\Big| \leq \frac{C(n,s)}{r^{k+|\alpha|}} \text{osc}_{Q_{r}(t_0,x_0) \times [0, r)} W.$$
			\item For $Q_r(t_0,x_0) \subset Q_1$, 
			$$\max_{Q_{\frac{r}{2}}(t_0,x_0) \times [0, \frac{r}{2})} \big|W\big| \leq C(r,n,s)\|W\|_{L^2(Q_r(t_0,x_0)\times[0,r),y^a)}.$$ 
			\item For all $0\leq y<\frac{1}{2}$,
			$$\sup_{(t,x)\in Q_{\frac{1}{2}}} \big|W_y(t,x,y)\big| \leq C(n,s)\|W\|_{L^2(Q_1^*,y^a)}y.$$
		\end{enumerate}
	\end{lemma}

	\section{Proof of Main Theorem}\label{s:m}
Our proof is based on an appropriate adaptation of  compactness arguments, which has its roots in a fundamental work of Caffarelli in \cite{Ca}.  We will begin by proving the relevent compactness lemma.
\begin{lemma}\label{closeness}
	Let $U$ be a weak solution of (\ref{mainpr})  with
	\begin{align}\label{l2u}
		\int_{Q_1}U(t,x,0)^2 dt dx + \int_{Q^*_{1}}y^a U^2 dt dX \leq 1
	\end{align}
	Then, given any $\epsilon >0$, there exixts a $\delta=\delta(\epsilon,n,s,ellipticity)>0$  such that if 
	\begin{align}\label{delta}
		\int_{Q_1} f^2 dt dx \leq \delta^2, \hspace{2mm} \int_{Q_1^*} y^a |F|^2 dt dX \leq \delta^2,\hspace{1mm} \text{and} \hspace{1mm}  \omega_{A}(1) \leq \delta^2,
	\end{align}
	then there exists $V$ which solves weakly 
	\begin{equation*}
		\left\{
		\begin{aligned}{}
			\operatorname{div}\big(y^a \nabla V\big)&=y^a \partial_t V &\mathrm{in}~(-3/4,3/4) \times B^*_{1/2} \\
			\lim_{y \rightarrow 0}y^a V_y&=0&\mathrm{on}~(-3/4,3/4) \times B_{1/2} \times \{0\}
		\end{aligned}
		\right.
	\end{equation*}
	such that 
	\begin{align*}
		\int_{Q^*_{1/2}}y^a|U-V|^2 dt dX < \epsilon^2 \hspace{2mm}\text{and} \hspace{2mm} \int_{Q_{1/2}}|U-V|^2 dt dx < \epsilon^2.
	\end{align*}
\end{lemma}

\begin{proof}
	\noindent Let $V$ be the solution of the following problem:
	\begin{equation}\label{Vproblem}
		\left\{
		\begin{aligned}{}
			\text{div}\big(y^a \nabla V\big)&=y^a \partial_t V &\mathrm{in}~(-1/2,3/4) \times B^*_{1/2} \\
			V &= U&\mathrm{on}~\partial_{p} ((-1/2,3/4) \times B^*_{1/2}) \setminus \big\{ (-1/2,3/4) \times B_{1/2} \times \{0\}\big\}\\
			\lim_{y \rightarrow 0}y^a V_y&=0&\mathrm{on}~(-1/2,3/4) \times B_{1/2} \times \{0\}.
		\end{aligned}
		\right.
	\end{equation}
	\noindent  Note that we have the existence and uniqueness of the above problem by Theorem \ref{exis}. Now using the regularity of $V$ (Lemma \ref{reg}) and the properties of Steklov average, we have, the Steklov average of $V,$ i.e., $V_h(t,x,y)=\fint_t^{t+h}V(\tau,x,y) d\tau ,$ $h<1/4,$ solves
	\begin{equation}\label{vh}
		\left\{
		\begin{aligned}{}
			\text{div}\big(y^a \nabla V_h\big)&=y^a \partial_t V_h &\mathrm{in}~Q^*_{1/2} \\
			V_h &= U_h&\mathrm{on}~\partial_{p} Q_{1/2}^{*} \setminus \left\{Q_{1/2} \cup \{t=-1/2\}\right\}\\
			\lim_{y \rightarrow 0}y^a \big(V_{h}\big)_y&=0&\mathrm{on}~Q_{1/2}.
		\end{aligned}
		\right.
	\end{equation}
	Now, multiplying \eqref{vh} by $\big(U_h - V_h\big)$ ($U_h$ is Steklov average of $U$) and integrate by parts to get
	\begin{equation}\label{20JULY3}
		- \int_{Q^*_{1/2}} y^a  \nabla V_h \nabla \big(U_h - V_h\big)dtdX -  \int_{Q^*_{1/2}} y^a  \partial_t V_h  \big(U_h - V_h\big)dtdX = 0.
	\end{equation}
	  By standard arguments, we know $U$ can be replaced by $U_h$ in  (\ref{weaksoldef}). Also $V_h$ solves (\ref{vh}) so we can put $\phi = \big(U_h - V_h\big)\eta(t)$ in (\ref{weaksoldef}), where $\eta(t)$ is choosen such that $\eta(t)=1$ in $(-1/2,1/2)$ and compactly supported in $(-5/8,5/8)$. Hence, after integrating by parts with respect to $t$ we get
	\begin{equation}\label{20JULY4}
		\begin{aligned}
			\int_{Q^*_{1/2}} y^a  \partial_t U_h \big(U_h - V_h\big)dtdX +  \int_{Q^*_{1/2}} y^a B(x) \nabla U_h \nabla \big(U_h - V_h\big)dtdX &= \int_{Q_{1/2}}   f_h \big(U_h - V_h\big)dtdx\\& +  \int_{Q^*_{1/2}} y^a  F_h \nabla \big(U_h - V_h\big)dtdX.
		\end{aligned}
	\end{equation}
	On adding \eqref{20JULY3} and \eqref{20JULY4}, we get
	\begin{equation}
		\begin{aligned}
			&\int_{Q^*_{1/2}} y^a  \partial_t \big(U_h - V_h\big)  \big(U_h - V_h\big)dXdt +  \int_{Q^*_{1/2}} y^a \big(B(x) - I\big)  \nabla U_h \nabla \big(U_h - V_h\big)dXdt \\&+  \int_{Q^*_{1/2}} y^a \big|\nabla \big(U_h - V_h\big)\big|^{2}dXdt
			= \int_{Q_{1/2}}   f_h \big(U_h - V_h\big)dxdt +  \int_{Q^*_{1/2}} y^a  F_h \nabla \big(U_h - V_h\big)dXdt.
		\end{aligned}
	\end{equation}
	\noindent Using Fundamental theorem of calculus in $t$ variable and Young's inequality with $\varepsilon$ we get, 
	\begin{equation}
		\begin{aligned}
		I_h(0) +	\int_{Q^*_{1/2}} y^a \big|\nabla \big(U_h - V_h\big)\big|^{2}dXdt\leq  \frac{1}{2}&\int_{Q^*_{1/2}} y^a \big|\nabla \big(U_h - V_h\big)\big|^{2} dXdt+ \frac{1}{2}\int_{Q^*_{1/2}} y^a \big|B(x) - I\big|^2  \big|\nabla U_h\big|^2 dXdt\\&+ \frac{1}{4\varepsilon}\int_{Q_{1/2}}   f_h^2dxdt
			+ \varepsilon\int_{Q_{1/2}} \big(U_h - V_h\big)^2dxdt\\ &+  \int_{Q^*_{1/2}} y^a  F_h^2dXdt+  \frac{1}{4} \int_{Q^*_{1/2}} y^a   \big|\nabla \big(U_h - V_h\big)\big|^2dXdt,
		\end{aligned}
	\end{equation}
where $I_h(0)=\int_{B^*_{1/2}} y^a \big(U_h - V_h\big)(0)dX.$ We rewrite above equation as 
\begin{align}\label{main}
 I_h(0)+	\frac{1}{4}	\int_{Q^*_{1/2}} y^a \big|\nabla \big(U_h - V_h\big)\big|^{2}dXdt\leq &\int_{Q^*_{1/2}} y^a \big|B(x) - I\big|^2  \big|\nabla U_h\big|^2 dXdt
		+ \varepsilon\int_{Q_{1/2}} \big(U_h - V_h\big)^2dxdt\\ &+\frac{1}{4\varepsilon}\int_{Q_{1/2}}   f_h^2dxdt+  \int_{Q^*_{1/2}} y^a  F_h^2dXdt
\end{align}
Note that we can make third and fourth term small by taking $\delta$ small enough. For the first term, take $\varphi$ compactly supported in ${Q^*_{3/4}}$ such that $\varphi=1$ in ${Q^*_{1/2}}$ and use Lemma \ref{energy} to get
\begin{align*}
	\frac{1}{2}&\int_{Q^*_{1/2}} y^a \big|B(x) - I\big|^2  \big|\nabla U_h\big|^2 dXdt\\
	 \leq &\frac{\omega_{A}(1)}{2}\int_{Q^*_{1/2}} y^a  \big|\nabla U_h\big|^2 dXdt\\
	\leq & C\frac{\omega_{A}(1)}{2}\left[\int_{-1/2}^{1/2}\int_{B^*_{1}} y^a(U_h^2+F_h^2) dX dt +\int_{-1/2}^{1/2}\int_{B_{1}} |U_h(t,x,0)||f_h(t,x)| dx dt\right]\\
	\leq & C\frac{\omega_{A}(1)}{2}\left[\int_{-1/2}^{1/2}\int_{B^*_{1}} y^a(U_h^2+F_h^2) dX dt +\frac{1}{2}\int_{-1/2}^{1/2}\int_{B_{1}} |U_h(t,x,0)|^2 dx dt+\frac{1}{2}\int_{-1/2}^{1/2}\int_{B_{1}}|f_h(t,x)|^2 dx dt\right]\\
	\leq & C\frac{\omega_{A}(1)}{2}\left[\int_{Q^*_{1}} y^a(U^2+F^2) dX dt +\frac{1}{2}\int_{Q_{1}} |U(t,x,0)|^2 dx dt+\frac{1}{2}\int_{Q_{1}}|f(t,x)|^2 dx dt\right]\\
	\leq & C \omega_{A}(1)(1+\delta^2),
\end{align*}
where $C=C(n,s, \text{ellipticity})>0.$ Also by trace theorem \cite{Nek}, we have 
\begin{align}\label{f2}
	\int_{Q_{1/2}} \big(U_h - V_h\big)^2dxdt \leq C_T\int_{Q^*_{1/2}} y^a\big(U_h - V_h\big)^2dXdt+C_T \int_{Q^*_{1/2}} y^a|\nabla\big(U_h - V_h\big)|^2dXdt
\end{align}
where $C_T(n,s)$ is constant from trace theorem.
Hence, (\ref{main}) becomes 
\begin{align}\label{f1}
 I_h(0)+	\frac{1}{4}\int_{Q^*_{1/2}} y^a \big|\nabla \big(U_h - V_h\big)\big|^{2}dXdt
	\leq  &C \omega_{A}(1)(1+\delta^2) + \varepsilon C_T\int_{Q^*_{1/2}} y^a\big(U_h - V_h\big)^2dXdt\\
	&+ \varepsilon C_T \int_{Q^*_{1/2}} y^a|\nabla\big(U_h - V_h\big)|^2dXdt +\frac{\delta^2}{4\varepsilon} +\delta^2.
\end{align}
Also, for each time label we will apply Poincare inequality \cite{FKS} to get 
\begin{align*}
 I_h(0)+	((1/4-\varepsilon C_T)C_P -\varepsilon C_T)\int_{Q^*_{1/2}} y^a  \big(U_h - V_h\big)^{2}dXdt \;
	\leq  \;C \omega_{A}(1)(1+\delta^2) + 
	\frac{\delta^2}{4\varepsilon} +\delta^2
\end{align*}
where $C_P(n,s)$ is constant in Poincare inequality.\\
Now, first choose $\varepsilon >0$ small enough such that $(1/4-\varepsilon C_T)C_P -\varepsilon C_T) >0$ than choose $\delta$ to get
$$ I_h(0)+\int_{Q^*_{1/2}} y^a  \big(U_h - V_h\big)^{2}dXdt \leq \epsilon^2.$$
Also, from (\ref{f2}) and (\ref{f1}) we have 
$$	\int_{Q_{1/2}} \big(U_h - V_h\big)^2dxdt \leq M \epsilon^2 +M I_h(0)$$
where $M$ is a universal constant. Since $U \in C([-1,1],L_a^2(B_1^*)) \cap L^2([-1,1],H_a^1(B_1^*))$, $U_h-V_h \rightarrow U-V$, $\nabla(U_h-V_h) \rightarrow \nabla(U-V)$ in $L^2_a(Q_{1/2}^*)$ and using \ref{Vproblem}, we have $I_h(0) \rightarrow 0$. Hence by trace theorem $(U_h-V_h)(t,x,0) \rightarrow (U-V)(t,x,0)$ in $L^2(Q_{1/2}).$ This completes the proof of the lemma.
\end{proof}
\begin{lemma}\label{linear}
	 There exist $0< \delta, \lambda <1$ (depending on $n$, $s$ and ellipticity), 
	a linear function $\ell(x) =\mathcal{A} + \mathcal{B} \cdot x$ and constant $C =C(n,s)>0$
	such that for any solution $U$ of (\ref{mainpr}) which satisfies (\ref{l2u}) and (\ref{delta}),
	$$\frac{1}{\lambda^{n+2}} \int_{Q_{\lambda}} |U(t,x,0)-\ell(x)|^2 \, dt \, dx + \frac{1}{\lambda^{n+3+a}} \int_{Q^*_{\lambda}} y^a|U-\ell(x)|^2 \, dt \, dX < \lambda^3 $$
	and $|\mathcal{A}| + |\mathcal{B}| \leq C.$
\end{lemma}

\begin{proof}
	Let $0<\epsilon <1$ be any real number. From Lemma \ref{closeness}, there exists a $\delta(\epsilon) >0$ such that if (\ref{delta}) holds, then 
	\begin{align}\label{clo}
		\int_{Q^*_{1/2}}y^a|U-V|^2 dt dX < \epsilon^2 \hspace{2mm} \text{and} \hspace{2mm} \int_{Q_{1/2}}|U-V|^2 dt dx < \epsilon^2
	\end{align}
	 where $V$ solves (\ref{Vproblem}). 
	 	Now  by Lemma \ref{reg}, $V$ is smooth, we define $$\ell(x) = V(0,0,0) + \nabla_x V(0,0,0) \cdot x=\mathcal{A}+\mathcal{B}\cdot x.$$
	 	Also by Lemma \ref{reg}, there exists a constant $\tilde{C}=\tilde{C}(n,s)$ such that 
	 	$$ |V(0,0,0)| + |\nabla_x V(0,0,0)| \leq \tilde{C} ||V||_{L^2_a(Q_{1/2} \times [0,1/2])}.$$
	 Now, using triangle inequality, $(a+b)^2 \leq 2a^2 +2b^2$, (\ref{clo}) and (\ref{l2u}), we get
	$$\int_{Q^*_{1/2}} y^a|V|^2  \, dt \, dX \leq 2\int_{Q^*_{1/2}}y^a|U-V|^2 dt dX +2\int_{Q^*_{1/2}} y^a|U|^2  \, dt \, dX \leq 2\epsilon^2+2 \leq 4.$$
	Hence we get
	$$ |V(0,0,0)| + |\nabla_x V(0,0,0)| \leq C, $$
	where $C=4 \tilde{C}.$\\
	Using triangle inequality, we get
	\begin{align*}
		\frac{1}{\lambda^{n+2}} &\int_{Q_{\lambda}} |U(t,x,0)-\ell(x)|^2\, dt \, dx + \frac{1}{\lambda^{n+3+a}} \int_{Q^*_{\lambda}} y^a|U-\ell(x)|^2 \, dt \, dX\\
		\leq & \frac{2}{\lambda^{n+3+a}} \int_{Q^*_{\lambda}} y^a|U-V|^2  \, dt \, dX + \frac{2}{\lambda^{n+3+a}} \int_{Q^*_{\lambda}} y^a|V-\ell(x)|^2  \, dt \, dX\\
		&+ \frac{2}{\lambda^{n+2}} \int_{Q_{\lambda}} |U-V|^2  \, dt \, dx + \frac{2}{\lambda^{n+2}} \int_{Q_{\lambda}} |V-\ell(x)|^2  \, dt \, dx.
	\end{align*}
	 For any $(t,x,y) \in Q^*_{1/4}$, using Mean Value theorem, Taylor's Theorem and Lemma \ref{reg}, we have a universal constant $D$ such that
	\begin{align*}
		|V(t,x,y)-\ell(x)| &\leq |V(t,x,y)-V(t,x,0)| + |V(t,x,0)-V(0,x,0) | \\ 
		&\quad+| V(0,x,0) - V(0,0,0) - \nabla_x V(0,0,0) \cdot x| \\ 
		&\leq |V_y(t,x,\xi_1)|y +|V_t(\xi_2,x,0)| t+D|x|^2 \hspace{2mm} \text{(for some $\xi_1 \in (0,y)$ and $\xi_2 \in (0,t)$)} \\
		&\leq D\xi y + Dt+D|x|^2\leq D(|X|^2 + t).
	\end{align*}
 Hence, using (\ref{clo}) we get
	\begin{align*}
			&\frac{1}{\lambda^{n+2}} \int_{Q_{\lambda}} |U(t,x,0)-\ell(x)|^2\, dt \, dx+\frac{1}{\lambda^{n+3+a}} \int_{Q^*_{\lambda}} y^a|U-\ell(x)|^2  \, dt \, dX \\ 
		&\leq   \frac{2 \epsilon^2 }{\lambda^{n+3+a}} + \frac{4 D^2 }{\lambda^{n+3+a}} \int_{Q^*_{\lambda}}y^a (|X|^4+ |t|^2) \, dt \, dX + \frac{2 \epsilon^2 }{\lambda^{n+2}} + \frac{4 D^2 }{\lambda^{n+2}} \int_{Q_{\lambda}} (|x|^4+ |t|^2) \, dt \, dx\\ 
		&\leq \frac{2 \epsilon^2 }{\lambda^{n+3+a}} + \frac{8D^2}{1+a}\lambda^4 +\frac{2 \epsilon^2 }{\lambda^{n+2}} +8D^2 \lambda^4.
	\end{align*}
	Now first choose $0<\lambda <1/4 $ small enough such that 
	\begin{align*}
		 \frac{8D^2}{1+a}\lambda^4 +8D^2 \lambda^4 \leq \frac{\lambda^3}{2},
	\end{align*}
    then take $\epsilon$ such that
    $$\epsilon^2=\frac{\lambda^{n+6+a}}{8},$$
    which in turn fixes $\delta.$
\end{proof}
Before moving ahead we would like to give some real analysis estimates. Following estimate is analogous to [\cite{DKM}, Equation (3.4)] \\
\textbf{Estimate 1:}
For given $0<\sigma<1$, define
$$r_i:=\frac{\sigma^i r}{2}$$
for $i=0,1,2,3,...$. We have 
\begin{align}\label{est1}
	\sum_{i=0}^{\infty}&r_i^{2s-1} \Big(\fint_{Q_{r_i}(x_0,t_0)}  |f(x,t)|^2 dx dt\Big)^{\frac{1}{2}}\leq c \int_{0}^{r}\rho^{2s-2}\Big(\fint_{Q_\rho(x_0,t_0)}  |f(x,t)|^2 dx dt\Big)^{\frac{1}{2}} d \rho
\end{align}
for some $c$ depends on $s$, $\sigma$ and $n$. 
\begin{proof}  In this proof we will denote $Q_{r_i}(x_0,t_0)$ by $Q_i$. We will assume $|Q_1|=1.$
\begin{align*}
	&\hspace{2mm}\sum_{i=0}^{\infty}r_i^{2s-1} \left(\fint_{Q_i} |f(x,t)|^2 dx dt\right)^{\frac{1}{2}}\\
	&=r_0^{2s-1}\left(\fint_{Q_0}  |f(x,t)|^2 dx dt\right)^{\frac{1}{2}}+\sum_{i=1}^{\infty}r_i^{2s-1}\left(\fint_{Q_i}  |f(x,t)|^2 dx dt\right)^{\frac{1}{2}}\\
	& =\frac{2s - 1}{2^{2s-1}-1}\int_{r/2}^{r}\rho^{2s-2} d\rho \left(\fint_{Q_0}  |f(x,t)|^2 dxdt\right)^{\frac{1}{2}}\\
	& \hspace{2mm}+ \sum_{i=1}^{\infty}\frac{\sigma^{2s -1} \big(2s - 1\big)}{1 -\sigma^{1-2s}} \int_{\sigma^i r/2}^{\sigma^{i-1} r/2}\rho^{2s-2} d\rho \left(\fint_{Q_i}  |f(x,t)|^2 dx dt\right)^{\frac{1}{2}}\\
	&=\frac{2s - 1}{2^{2s-1}-1}\int_{r/2}^{r}\rho^{2s-2} d\rho \frac{1}{( r/2)^{(n+2)/2}}\left(\int_{Q_0}  |f(x,t)|^2 dx  dt\right)^{\frac{1}{2}}\\ & \hspace{2mm}+  \sum_{i=1}^{\infty}\frac{\sigma^{2s -1} \left(2s - 1\right)}{1 -\sigma^{1-2s}}  \int_{\sigma^i r/2}^{\sigma^{i-1} r/2}\rho^{2s-2} d\rho \frac{1}{(\sigma^i r/2)^{(n+2)/2}}\left(\int_{Q_i}  |f(x,t)|^2 dx dt\right)^{\frac{1}{2}}\\
	& \leq  \frac{(2s - 1)2^{\frac{(n+2)}{2}}}{\left(2^{2s-1}-1\right)}\int_{r/2}^{r}\rho^{2s-2}  \frac{1}{ \rho^{(n+2)/2}}\left(\int_{Q_\rho}  |f(x,t)|^2 dx dt\right)^{\frac{1}{2}} d\rho \\
	&+  \sum_{i=1}^{\infty}\frac{\sigma^{2s -1} \left(2s - 1\right)\sigma^{\frac{(n+2)}{2}}}{\left(1 -\sigma^{1-2s}\right) }  \int_{\sigma^i r/2}^{\sigma^{i-1} r/2}\rho^{2s-2}\frac{1}{\rho^{(n+2)/2}}\left(\int_{Q_\rho}  |f(x,t)|^2 dx dt\right)^{\frac{1}{2}} d \rho\\
	&\leq  \left(\frac{(2s - 1)2^{\frac{(n+2)}{2}}}{\left(2^{2s-1}-1\right)} +\frac{\sigma^{2s -1} \left(2s - 1\right)\sigma^{\frac{(n+2)}{2}}}{\left(1 -\sigma^{1-2s}\right) } \right) 
	\int_{0}^{r}\rho^{2s-2}\left(\fint_{Q_\rho}  |f(x,t)|^2 dx dt\right)^{\frac{1}{2}} d \rho.
\end{align*}\end{proof}
\textbf{Estimate 2:} We have the following estimate which is analogue of [\cite{DKM}, Equation (3.13)],
\begin{align}\label{est2}
	\sup_{(x,t)\in Q_{1/2}}\tilde{\mathbf{I}}_{2}^{f}((x,t),r)\;\leq \;\frac{1}{(n + 2)\;C_{n + 2}^{(2s - 1)/(n + 2)}}\int_{0}^{C_{n + 2} r^{n + 2} }\;\rho^{\frac{2s - 1}{n + 2}}  (g^{**}(\rho))^{\frac{1}{2}}\frac{d\rho}{\rho},
\end{align}
where $0<r<1/2$,  $\tilde{\mathbf{I}}_{2}^{f}((x,t),r)=\int_{0}^{r}\rho^{2s-2}\Big(\fint_{Q_\rho(x_0,t_0)}  |f(x,t)|^2 dx dt\Big)^{\frac{1}{2}} d \rho$, $g=|f|^2$ and $C_{n+2}$ denotes measure of $Q_1$ in $\R^{n+1}.$
\begin{proof}
	From Hardy-Littlewood inequality \cite{HL}, we have
	\begin{align}\label{171}
		\fint_{Q_{\rho}(x_0,t_0)} g(t,x)dtdx \;\leq \;\frac{1}{C_{n + 2} \rho^{n + 2}} \int_{0}^{C_{n + 2} \rho^{n + 2}} g^*(s)ds\;\leq \;g^{**}(C_{n + 2} \rho^{n + 2}).
	\end{align}
	Multiply \eqref{171} with $\rho^{2s - 2}$ and integrate with respect to $\rho$ in $(0,r)$ to get
	\begin{equation*}
		\tilde{\mathbf{I}}_{2}^{f}((x_0,t_0),r)\;\leq \;\int_{0}^{r}\;\rho^{2s - 2}(g^{**}(C_{n + 2}\rho^{n + 2}))^{\frac{1}{2}}\;d\rho.
	\end{equation*}
Now by changing variables we get (\ref{est2}).	
\end{proof}

Now we will define some functions as in \cite{Ak}:
\begin{itemize}
	\item Define
\begin{align*}
	\tilde{\omega}_1(r):=\text{max}\{\omega_{A}(\gamma r)/\tilde{\delta},r\},
\end{align*}
 where $\gamma$ and $\tilde{\delta}$ will be fixed later. By [\cite{L} Theorem 8] we can assume $\omega_1$ is concave. Without loss of generality we can assume ${\omega}_1(1)=1$. Finally we define
\begin{align*}
	\omega_1(r):=\tilde{\omega}_1(\sqrt{r}).
\end{align*}
Then, $\omega_1(r)$ is $1/2$-decreasing function. Also, by application of change of variables we have $\omega_1(r)$ is Dini continuous.
\item We define
\begin{align*}
	\omega_2(r):=max\left\{\gamma I(\gamma r)/\tilde{\delta},r\right\},
\end{align*}
where $	I(r):= r^{2s-1}\bigg(\fint_{Q_r}  |f(x,t)|^2 dt dx\bigg)^{1/2}.$
\item We define
\begin{align*}
	\omega_3(\lambda^k):=\sum_{i=0}^{k}	\omega_1(\lambda^{k-i})	\omega_2(\lambda^i),
\end{align*}
where $\lambda$ is from Lemma \ref{linear}.
\item Finally, we define
\begin{align*}
	\omega(\lambda^k)=max\left\{\omega_3(\lambda^k),\lambda^{k/2}\right\}.
\end{align*}
\end{itemize}
\begin{lemma}
There exists a universal constant $C_{sum}$ such that
\begin{align*}
		\sum_{i=0}^{\infty}\omega(\lambda^i) \le C_{sum}.
\end{align*}
\end{lemma}
\begin{proof}	
Since $\omega_1$ is a increasing function, we get 
\begin{align*}
	\sum_{i=0}^{\infty}\omega(\lambda^i) \le \sum_{i=0}^{\infty}\omega_3(\lambda^i)+\frac{1}{1-\sqrt{\lambda}}\le\sum_{i=0}^{\infty}\omega_1(\lambda^i)\sum_{i=0}^{\infty}\omega_2(\lambda^i)+\frac{1}{1-\sqrt{\lambda}}.
\end{align*}
Using the Dini continuity, increasing nature of $\omega_A$ and change of variables, we have
\begin{align*}
	\sum_{i=0}^{\infty}\omega_1(\lambda^i) \le \sum_{i=0}^{\infty}\frac{\omega_A(\gamma\lambda^{i/2})}{\tilde{\delta}}+\frac{1}{1-\sqrt{\lambda}} \le \frac{\omega_A(\gamma)}{\tilde{\delta}} + \frac{1}{(-\log \sqrt{\lambda})\tilde{\delta} }\int_{0}^{\gamma}\frac{\omega_A(s)}{s}ds.
\end{align*}
Choose $\gamma < \tilde{\delta}$ such that 
\begin{align*}
		\sum_{i=0}^{\infty}\omega_1(\lambda^i) \le 1.
\end{align*}
Note that from (\ref{est1}), we have
\begin{align*}
	\sum_{i=0}^{\infty}\omega_2(\lambda^i) \le \frac{c\gamma}{\tilde{\delta}}\tilde{\mathbf{I}}_{2}^{f}((0,0),\gamma)+\frac{1}{1-\lambda}.
\end{align*}
Using (\ref{est2}) we have
\begin{align*}
	\sum_{i=0}^{\infty}\omega_2(\lambda^i) \le \frac{c\gamma}{\tilde{\delta}(n + 2)\;C_{n + 2}^{(2s - 1)/(n + 2)}}\int_{0}^{C_{n + 2} \gamma^{n + 2} }\rho^{\frac{2s - 1}{n + 2}}  (g^{**}(\rho))^{\frac{1}{2}}\frac{d\rho}{\rho},
\end{align*}
where $g=f^2$. We have given that $f \in L\big(\frac{n+2}{2s-1},1\big)$, which gives us $g \in L\big(\frac{n+2}{2(2s-1)},\frac{1}{2}\big)$. Since $2(2s-1) < n+2$, therefore from [\cite{ON}, Equation (6.8)], we get
\begin{align*}
	\int_{0}^{\infty} \rho^{\frac{2s - 1}{n + 2}} (g^{**}(\rho))^{\frac{1}{2}}\frac{d\rho}{\rho} < \infty.
\end{align*}
Hence we get  a universal bound on $	\sum_{i=0}^{\infty}\omega_2(\lambda^i).$ This completes the proof of the lemma.
\end{proof}
 Note that by replacing $U(t,X)$ by 
 $$U_{\gamma}(t,X)=\left(\fint_{Q_\gamma}U(t,x,0)^2dtdx+\fint_{Q^*_\gamma}U(t,X)^2dtdX+1\right)^{-1} U(\gamma^2t,\gamma X),$$
 we can assume 
 $$\fint_{Q_1}U(t,x,0)^2dtdx+\fint_{Q^*_1}U(t,X)^2dtdX \le 1,$$
 modulus of continuity for $A$ is $\omega_A(\gamma r),$ and $r^{2s-1}\bigg(\fint_{Q_r}  |f(x,t)|^2 dt dx\bigg)^{1/2}\le \gamma I(\gamma r)$. In the following lemma we shall denote $U_{\gamma}(t,X)$ by $U(t,X)$, $\omega_A(\gamma r)$ by $\omega_A(r),$ and $\gamma I(\gamma r)$ by $I(r)$.
  \begin{lemma}\label{lseq}
	There exist a sequence of linear functions $\ell_{k}(x) =a_k + b_k \cdot x$ and a constant $C=C(n,s)>0$
	such that for any solution $U$ of (\ref{mainpr}) satisfies
	\begin{align}\label{ind}
		\frac{1}{\lambda^{k(n+2)}} \int_{Q_{\lambda^k}} |U(t,x,0)-\ell_{k}(x)|^2 \, dt \, dx + \frac{1}{\lambda^{k(n+3+a)}} \int_{Q^*_{\lambda^k}} y^a|U-\ell_{k}(x)|^2 \, dt \, dX < \lambda^{2k}{\omega}^2(\lambda^k),
	\end{align}
\begin{align}\label{akk1}
   |a_{k+1}-a_k| \leq C \lambda^k \omega(\lambda^k) \hspace{2mm} \text{and} \hspace{2mm} |b_{k+1}-b_k| \leq C \omega(\lambda^k).
\end{align}
	
\end{lemma}
\begin{proof}
	We will prove the lemma by induction on $k$. For $k=0$, take $\ell_{k}(x)=0$ and we get 
		$$ \int_{Q_1} |U(t,x,0)|^2 \, dt \, dx + \int_{Q^*_1} y^a|U|^2 \, dt \, dX \leq 1 \leq {\omega}^2(1). $$ Thus we are done for $k=0$.
		
		Let us assume result is true for $k=0,1,...,i$. We will prove it for $k=i+1$. Define
		$$\tilde{U}(t,X):=\frac{U(\lambda^{2i}t,\lambda^i X)-\ell_{i}(\lambda^i x)}{\lambda^{i}{\omega}(\lambda^i)},$$
	    then $\tilde{U}$ is a weak solution to 
		$$\begin{cases}
			y^a \partial_t \tilde{U} - \text{div}(y^a \tilde{B} \nabla \tilde{U}) = - \text{div}(y^aF )&\hbox{in}~Q^*_1 \\
			-y^a\tilde{U}_y|_{y=0}=\tilde{f} &\hbox{on}~Q_1,
		\end{cases}$$
		where $\tilde{B}(X)=B(\lambda^i X)$, $\tilde{f}(t,x)=\frac{\lambda^{-ia}}{{\omega}(\lambda^i)}f(\lambda^{2i} t, \lambda^i x)$ and the vector field $F$ is given by ${\lambda^{i}{\omega}(\lambda^i)}F= ((I-A(\lambda^i x))\nabla_x \ell(\lambda^i x), 0 )$. Now we estimate the following:\\
		\textbf{Bound for $f$}:
		\begin{align*}
			\int_{Q_1} \tilde{f}^2 dt dx &=\frac{\lambda^{-2ia}}{{\omega}^2(\lambda^i)} \int_{Q_1}f^{2}(\lambda^{2i} t, \lambda^i x)  dt dx= \frac{\lambda^{-2ia}}{{\omega}^2(\lambda^i)} \fint_{Q_{\lambda^i}}f^{2}( t, x)  dt dx\\
			&\le\frac{{\omega_2^2}(\lambda^{i})}{{\omega}^2(\lambda^i)}
			\leq {\tilde{\delta}}^2\frac{{\omega_2^2}(\lambda^{i})}{\omega_1^2(1){\omega_2}^2(\lambda^i)}={\tilde{\delta}}^2.
		\end{align*}
	\textbf{Bound for $F$:} Note that by induction hypothesis\\ $$|b_i| \le \sum_{j=1}^{i}|b_j-b_{j-1}| \le \sum_{j=0}^{i-1}C \omega(\lambda^i) \le CC_{sum}=:C_1.  $$ Now
	\begin{align*}
		\int_{Q_1^*} y^a |F|^2 dtdX
		 &=\lambda^{-2i}{\omega}^{-2}(\lambda^i)\int_{Q_1^*}y^a|(I-A(\lambda^i x))\nabla_x \ell_i(\lambda^i x)|^2 dtdX \\
		 &=2{\omega}^{-2}(\lambda^i)\fint_{B_{\lambda^i}}|(I-A(x))b_i|^2 dtdX\\
		 &\leq 2{\omega}^{-2}(\lambda^i)\omega_{A}^2(\lambda^i)C_1^2\\
		 &\leq 2C^2_1\frac{\omega_{A}^2(\lambda^i)}{\omega_{1}^2(\lambda^i)\omega_2(1)}
		 \le 2C^2_1\tilde{\delta}^2\frac{\omega_{A}^2(\lambda^i)}{\omega_{A}^2(\lambda^i)}=2C^2_1\tilde{\delta}^2.
	\end{align*}
\textbf{	Bound for $\tilde{B}$:}
		\begin{align*}
			|B(\lambda^i x)-I|^2 \leq \omega_A^2(\lambda^i) \leq \omega_A^2(1) \le \tilde{\delta}^2. 
		\end{align*} 
	Choose $\tilde{\delta}<\delta$ such that $2C^2_1\tilde{\delta}^2<\delta^2$.
	We rewrite (\ref{ind}) for $k=i$ as 
	\begin{align*}
		\frac{1}{\lambda^{i(n+2)}} \int_{Q_{\lambda^i}} \frac{|U(t,x,0)-\ell_{i}(x)|^2}{\lambda^{2i}{\omega}^2(\lambda^i)} \, dt \, dx + \frac{1}{\lambda^{i(n+3+a)}} \int_{Q^*_{\lambda^i}} y^a\frac{|U(t,x,y)-\ell_{i}(x)|^2}{\lambda^{2i}{\omega}^2(\lambda^i)} \, dt \, dX <1.
	\end{align*}
	Now after changing the variables, we get	
	$$\int_{Q_{1}}|\tilde{U}(t,x,0)|^2 dt dx +\int_{Q^*_{1}}y^a|\tilde{U}(t,x,y)|^2 dt dX \leq 1.$$
	By Lemma \ref{linear} there exists a $l(x)=a+b\cdot x$ such that
		$$\frac{1}{\lambda^{n+2}} \int_{Q_{\lambda}} |\tilde{U}(t,x,0)-l(x)|^2 \, dt \, dx + \frac{1}{\lambda^{n+3+a}} \int_{Q^*_{\lambda}} y^a|\tilde{U}-l(x)|^2 \, dt \, dX < \lambda^{3} $$
	and $|a| + |b| \leq C.$
	After putting the value of $\tilde{U}$ and changing the variables, we get 
	\begin{align*}
		\frac{1}{\lambda^{(i+1)(n+2)}} \int_{Q_{\lambda^{i+1}}} |U(t,x,0)-\ell_{i+1}(x)|^2 \, dt \, dx + \frac{1}{\lambda^{(i+1)(n+3+a)}} \int_{Q^*_{\lambda^{i+1}}} y^a|U-\ell_{i+1}(x)|^2 \, dt \, dX < \lambda^{2(i+1)}{\omega}^2(\lambda^{i+1})
	\end{align*}
	where $\ell_{i+1}(x)=\ell_i(x)+\lambda^{i}\omega(\lambda^i)l(\lambda^{-i}x).$
	By putting $x=0$ we get $|a_{i+1}-a_i| \leq C \lambda^{i}\omega(\lambda^i)$. Take gradient to get $|b_{i+1}-b_i| \leq C \omega(\lambda^i)$. 
\end{proof}

\begin{lemma}\label{l2}
	There exists a linear function $\ell_{t_0,x_0}(x)=a_{\infty}+b_{\infty}\cdot x$ such that for any solution $U$ of (\ref{mainpr}) with (\ref{l2u}) satisfies
	\begin{align*}
	\frac{1}{r^{n+3+a}}	\int_{Q_r^*(t_0,x_0,0)} |U(t,x,y)-\ell_{(t_0,x_0)}(x)|^2 y^a dt dX \leq C_lr^2K^2(r)
	\end{align*}
 for all $0<r<1/2$, where $C_l$ is a universal constant and $K$ is $1/2-$decreasing concave modulus of continuity.
\end{lemma}
\begin{proof} WLOG we will assume that $(t_0,x_0)=(0,0).$ First we will estimate for $U_{\gamma}$ with small $r$ then will do it for $U.$ Note that $a_k$ and $b_k$ are convergent as $$	\sum_{i=0}^{\infty}\omega(\lambda^i) < \infty.$$  We define $a_{\infty}$ and $b_{\infty}$ as limit of $a_k$ and $b_k$ respectively.   We will denote $\ell_{(0,0)}$ by $\ell_0$. Take $r$ such that $\lambda^{k+1} \leq r \leq \lambda^{k}$ for some $k \in \mathbb{N}.$ Now using triangle inequality, (\ref{ind}) and (\ref{akk1}) we have
	\begin{align*}
		\fint_{Q_r} |U(t,x,0)-\ell_{0}(x)|^2 \, dt \, dx
		\leq & 2\fint_{Q_r} |U(t,x,0)-\ell_k(x)|^2 \, dt \, dx +2\fint_{Q_r} |\ell_k-\ell_{0}(x)|^2 \, dt \, dx\\
		\leq & 2\lambda^{-n-2} \lambda^{2k} \omega^2(\lambda^k) +2\fint_{Q_r} \sum_{i=k}^{\infty}|\ell_i-\ell_{i+1}|^2 \, dt \, dx\\
		\leq & 2\lambda^{-n-2} \lambda^{2k} \omega^2(\lambda^k) +4 \sum_{i=k}^{\infty}|a_i-a_{i+1}|^2 +4r^2\sum_{i=k}^{\infty}|b_i-b_{i+1}|^2\\
		\leq & 2\lambda^{-n-2} \lambda^{2k} \omega^2(\lambda^k) +4C\lambda^{2k}\sum_{i=k}^{\infty}\omega^2(\lambda^i)+4C\lambda^{2k}\sum_{i=k}^{\infty}\omega^2(\lambda^i)\\
		\leq &(8C+2)\lambda^{2k}\bigg(\sum_{i=k}^{\infty}\omega(\lambda^i)\bigg)^2
	\end{align*}
	Now follow the same lines of proof as in \cite{Ak} from $(4.30)$ to $(4.35)$ to get
	\begin{align}\label{ugam}
		\frac{1}{r^{n+3+a}}	\int_{Q_r^*} |U(t,x,y)-\ell_{0}(x)|^2 y^a dt dX \leq C_lr^2K^2(r),
	\end{align}
where $K=K_1(r)+K_2(r)+K_3(r)$ and 
$$K_1(r):=\underset{a \ge 0}{\text{sup}}\int_{a}^{a+\sqrt{r}}\frac{\omega_1(t)}{t}dt, \hspace{2mm} K_2(r):=\sqrt{r}\hspace{2mm} \text{and}\hspace{2mm}K_3(r)=\underset{a \ge 0}{\text{sup}}\int_{a}^{a+C_{n + 2}r }\rho^{\frac{2s - 1}{n + 2}}  (g^{**}(\rho))^{\frac{1}{2}}\frac{d\rho}{\rho}.$$
Now put the value of $U_{\gamma}$ and do change of variables to get
\begin{align}\label{ugam1}
	\fint_{Q_{r \gamma}^*} |U(t,x,y)-\ell_{0}(\gamma^{-1}x)|^2 y^a dt dX \leq {\gamma}^{-2}(1+||U(t,x,0)||_{L^2(Q_{\gamma})}+||U(t,X)||_{L^2(Q_{\gamma}^*)})C_l(r \gamma )^2 K^2(r\gamma/{\gamma}).
\end{align}
Call $K(r/{\gamma})$ by $K(r)$, $\ell_{0}(\gamma^{-1}x)$ by $\ell_{0}(x)$ and use the fact $\gamma$ is a universal constant to get 
\begin{align}\label{ugam3}
		\fint_{Q_r^*} |U(t,x,y)-\ell_{0}(x)|^2 y^a dt dX \leq \tilde{C_l}r^2K^2(r),
\end{align}
for $0<r< \gamma \lambda.$  For $\gamma \lambda<r<1/2,$ we will done by replacing $\tilde{C_l}$ to $C\tilde{C_l}/(\gamma^{n+6}\lambda K(\gamma \lambda)),$ call it again by $C_l$.
\end{proof}
\begin{lemma}\label{lthin}
	There exists a universal constant $C_{bdr}$ such that 
	$$|\nabla \ell_{(t_1,x_1)}-\nabla \ell_{(t_2,x_2)}| \le C_{bdr}K(|(t_1,x_1)-(t_2,x_2)|),$$
	and $$|U(t_1,x,0)-U(t_2,x,0)| \le C_{bdr}K(\sqrt{|t_1-t_2|}) |t_1-t_2|^{1/2}$$
	for $(t_i,x_i) \in Q_{1/2}$, where $\ell_{(t_i,x_i)}$ denotes the linear function constructed in Lemma \ref{l2}. 
	
\end{lemma}
\begin{proof} Proceeding like previous lemma we get a linear function $\ell_{t_0,x_0}(x)=a_{\infty}+b_{\infty}\cdot x$ such that
	\begin{align*}
		\fint_{Q_r(t_0,x_0,0)} |U(t,x,0)-\ell_{(t_0,x_0)}(x)|^2 \, dt \, dx \leq Cr^2K^2(r).
	\end{align*}
Now	apply Companato Characterization to get the result.
\end{proof}
Now  to prove Theorem \ref{mthr}, we have to combine the above boundary esimates with known interior estimates in \cite{ck}. In order to do this we need the following rescaled version of interior estimate.
\begin{thrm}\label{irprop}
	Let $u$ be a weak solution of 
	\begin{align}\label{irp}
		u_t-div(B(x)\nabla u)-b\cdot \nabla u=div(g) \hspace{2mm} \text{in} \hspace{2mm} Q_1,
	\end{align}
with $||u||_{L^2(Q_1)} \le 1,$ where $B,$ $g$ are Dini continuous and $b$ is smooth. Then there exists a constant $C_{ir}>0$ depending on $n$, ellipticity, $||b||_1$ ($C^1$-norm of $b$), $||g||_{\infty},$ and $\Psi(1)$ such that for all $(t_1,x_1),(t_2,x_2) \in Q_{1/2},$ we have 
\begin{align*}
 |\nabla u(t_1,x_1)-\nabla u(t_2,x_2)|\le C_{ir}(\Psi(|(t_1,x_1)-(t_2,x_2)|)+|(t_1,x_1)-(t_2,x_2)|^{\alpha}),
\end{align*}
where $$\Psi(r)=\underset{a\ge 0}{\text{sup}}\int_{a}^{a+\sqrt{r}}\frac{\Psi_B(s)}{s}ds+\underset{a\ge 0}{\text{sup}}\int_{a}^{a+\sqrt{r}}\frac{\Psi_g(s)}{s}ds$$
and $$|u(t_1,x)-u(t_2,x)| \le C_{ir}(\Psi(\sqrt{|t_1-t_2|})+|t_1-t_2|^{1/4})\sqrt{|t_1-t_2|}$$
Furthermore for all $(t,x) \in Q_{1/2},$
\begin{align*}
	|\nabla u(t,x)|\le C_{ir}(1+\Psi(1)).
\end{align*}
\end{thrm}
The proof of the Theorem \ref{irprop} will be done by compactness method. We will sketch it as most of the details are similar to boundary estimate. We need the following lemmas.
\begin{lemma}\label{ircl}
	Let $u$ be a weak solution of (\ref{irp})
	with $||u||_{L^2(Q_1)} \le 1.$ 
	Then, given any $\epsilon >0$, there exixts a $\delta=\delta(\epsilon, n, \text{ellipticity},||b||_1,||g||_{\infty} ) >0$ such that if 
	\begin{align}\label{delint}
		\omega_{B}(1) \leq \delta^2 \hspace{2mm} \text{and} \hspace{2mm} \omega_{g}(1) \leq \delta^2.
	\end{align}
	then there exists $v$ which solves weakly 
	\begin{equation}\label{intexp}
		\left\{
		\begin{aligned}{}
			v_t-  \Delta v +b\cdot \nabla v&=0 \hspace{2mm} \text{in} \hspace{2mm}Q_{1/2}\\
			v&=u	\hspace{2mm} \text{on} \hspace{2mm}\partial_pQ_{1/2}
		\end{aligned}
		\right.
	\end{equation}
	such that 
	\begin{align*}
		\int_{Q_{1/2}}|u-v|^2 dt dX < \epsilon^2 .
	\end{align*}
\end{lemma}
\begin{proof}
	We have existence of solution of (\ref{intexp}) by Remark \ref{intexsit}.  Now proceed along the similar lines as in previous lemma. For the required Energy estimate see Lemma 1.2.3 (\cite{ck}). 
\end{proof}
\begin{lemma}\label{linearu}
	There exist $0< \delta_{int}, \lambda_{int} <1$ (depending on $n$, ellipticity, $||b||_1$),
	a linear function $\ell(x) =\mathcal{A}_{int} + \mathcal{B}_{int} \cdot x$ and constant $C_{int} =C(n,||b||_1)>0$
	such that for any solution $u$ of (\ref{irp}) with $||u||_{L^2(Q_1)}$ and satisfies (\ref{delint}),
	$$\frac{1}{\lambda^{n+2}} \int_{Q_{\lambda}} |u-\ell(x)|^2 dt dx < \lambda^3 $$
	and $|\mathcal{A}_{int}| + |\mathcal{B}_{int}| \leq C_{int}.$
\end{lemma}
\begin{proof}
	Since $v$, solution of (\ref{intexp}), is smooth. We can proceed like the proof of the Lemma \ref{linear}. 
\end{proof}

\begin{lemma}
	There exist a sequence of linear functions $\ell^{k}_{int}(x) =a^k_{int} + b^k_{int} \cdot x$ and a constant $C_{int}=C(n,||b||_1)>0$
	such that for any weak solution $u$ of (\ref{irp}) with $||u||_{L^2(Q_1)}$, 
	\begin{align}\label{indu}
		\frac{1}{\lambda_{int}^{k(n+2)}} \int_{{Q}_{\lambda_{int}^k}} |u-\ell^{k}_{int}(x)|^2 \, dt \, dx < \lambda^{2k}_{int}{\psi}^2(\lambda_{int}^k)
	\end{align}
and
	\begin{align}\label{akku}
		|a^{k+1}_{int}-a^k_{int}| \leq C_{int} \lambda^k_{int} \psi(\lambda_{int}^k) \hspace{2mm} \text{and} \hspace{2mm} |b^{k+1}_{int}-b^k_{int}| \leq C_{int} \psi(\lambda^k_{int})
	\end{align}
	
\end{lemma}
\begin{proof} Note that $u_{\gamma}(t,x)=u(\gamma^2 t, \gamma x)$  will satisfy (\ref{irp}) with $\Psi_B(\gamma r)$ and $\gamma g(\gamma x)$. Now we will some define functions:\\
	Define $$\tilde{\psi}_1(r):=max\{\Psi_B(\gamma_{int}r)/\tilde{\delta}_{int},r\},$$
	where $\gamma_{int}$ and $\tilde{\delta}_{int}$ will be fixed later. By [\cite{L} Theorem 8] we can assume $\tilde{\Psi}_1$ is concave. Without loss of generality we can assume $\tilde{\Psi}_1(1)=1$. Finally we define
	\begin{align*}
		\psi_1(r):=\tilde{\psi}_1(\sqrt{r}).
	\end{align*}
	Then, $\psi_1(r)$ is $1/2$-decreasing function. Also, using change of variables we have $\psi_1(r)$ is Dini continuous. 
	
	Similarly, define $\psi_2(r)$ corresponding to Dini conutinity of $g$, i.e., $\omega_g.$ Finally define $$\psi :=max\{\psi_3(r),\sqrt{r}\}$$
	where $\psi_3(\lambda^k):=\sum_{i=0}^{k}	\psi_1(\lambda^{k-i})	\psi_2(\lambda^i).$ 
	
	Now proceed as in the Lemma \ref{lseq}.
\end{proof}

\begin{proof}[Sketch of proof of Theorem \ref{irprop} ]
	First we will get the estimate as in Lemma \ref{l2} and then use Companato type characterization to get the implications of the theorem.
\end{proof}

\begin{proof}[Proof of Theorem \ref{mthr}]
	Let $(t_1,x_1,y_1)$ and $(t_2,x_2,y_2)$ be in $Q^*_{1/2}.$ Without loss of generality we shall assume $y_1 \leq y_2.$ We will do proof in two cases:
	\begin{itemize}
		\item[(i)]$|(t_1,X_1)-(t_2,X_2)| \leq \frac{y_1}{4}.$
		\item[(ii):] $|(t_1,X_1)-(t_2,X_2)| \geq \frac{y_1}{4}.$
	\end{itemize}
		In the first case, we will consider 
		$$\tilde{W}(t,x,y):=U(t,x,y)-\ell_{(t_1,x_1)}(x),$$
		then using Lemma \ref{l2} for $r=y_1/2$, we have
		\begin{align*}
			\int_{Q_{\frac{y_1}{2}}^*(t_1,x_1,y_1)} |U(t,x,y)-\ell_{(t_1,x_1)}(x)|^2 y^a dt dX \leq C_nC_l {y_1}^{n+5+a}K^2(y_1),
		\end{align*}
	where $C_n=2^{n+5+a}.$ Note that the following rescaled function
	\begin{align*}
		W(t,x,y)=\tilde{W}(t_1+y_1^2t,\; x_1+y_1x,\;y_1y)/y_1
	\end{align*}
solves $W_t-div(B(y_1X)\nabla W)-(a/y)W_y=-div ( (B(y_1X)-I )\cdot \nabla \tilde{\ell}),$ in $\textbf{B}_{1/2}$ and
 $$||W||_{L^2(\textbf{Q}_{1/2})} \le C_nC_lK(y_1) .$$ Now apply Proposition \ref{irprop} to get
    \begin{align*}
	|\nabla W(t,X)-\nabla W(0,0,1)| \le C_{ir}||W||_{L^2(\textbf{Q}_{1/2})}(K(y_1|(t,X)-(0,0,1)|)+|(t,X)-(0,0,1)|^{1/2}).
    \end{align*}
 Using $\ell_{(t_1,x_1)}$ is linear and rescaling back to get 
   \begin{align*}
	|\nabla U(t_1,X_1)-\nabla U(t_2,X_2)| &\leq C(K(y_1) K(|(t_1,X_1)-(t_2,X_2)|)+\frac{K(y_1)}{\sqrt{y_1}}|(t_1,X_1)-(t_2,X_2)|^{1/2})
   \end{align*}
Use $K$ is $1/2-$decreasing and bounded to get 
 \begin{align*}
	|\nabla U(t_1,X_1)-\nabla U(t_2,X_2)| &\leq C K(|(t_1,X_1)-(t_2,X_2)|).
\end{align*}
Similary, we get
$$|U(t_1,X)-U(t_2,X)| \le CK(|t_1-t_2|^{1/2})|t_1-t_2|^{1/2}.$$
This proves the first case.

To prove the case(ii) first note that using triangle inequality we get
\begin{align*}
	|y_2| &=|(t_2,X_2)-(t_2,x_2,0)|\le |(t_2,X_2)-(t_1,X_1)|+|(t_1,X_1)-(t_1,x_1,0)|+|(t_1,x_1,0)-(t_2,x_2,0)|\\
	&\le |(t_2,X_2)-(t_1,X_1)|+|y_1|+ |(t_2,X_2)-(t_1,X_1)|\le 6|(t_2,X_2)-(t_1,X_1)|.
\end{align*}
Use Lemma \ref{lthin} and Lemma \ref{l2} to get the following estimate
\begin{align*}
	|\nabla U(t_1,X_1)-\nabla U(t_2,X_2)|
	\le& |\nabla U(t_1,X_1)-\nabla \ell_{(t_1,x_1)}|+|\nabla \ell_{(t_1,x_1)}-\nabla \ell_{(t_2,x_2)}|+|\nabla U(t_2,X_2)-\nabla \ell_{(t_2,x_2)}|\\
	\le &CK(y_1)+CK(|(t_1,x_1)-(t_2,x_2)|)+CK(y_2)
	\le CK(6|(t_1,X_1)-(t_2,X_2)|).
\end{align*} 
Now we will get estimate for continuity in $t$ variable:
\begin{align*}
	|U(t_1,X)-U(t_2,X)| &\le |U(t_1,x,y)-U(t_1,x,0)+U(t_2,x,0)-U(t_2,x,y)|+|U(t_1,x,0)-U(t_2,x,0)|\\
	&\le|U_y(t_1,x,\xi_1)-U_y(t_2,x,\xi_2)||y| + CK(|t_1-t_2|^{1/2})|t_1-t_2|^{1/2} \\
	&\le C(K(|(t_1,x,\xi_1)-(t_1,x,\xi_2)|)+K(|t_1-t_2|^{1/2}))|t_1-t_2|^{1/2},
\end{align*}
where $\xi_1,\xi_2$ lies between $0$ and $y.$ Using the condition of Case(ii) we get 
 $$|U(t_1,X)-U(t_2,X)| \le CK(6|t_1-t_2|^{1/2})|t_1-t_2|^{1/2}.$$
This completes the proof of case(ii). \end{proof}
\section{Appendix}\label{a}
	\begin{thrm}\label{exis} For every $f \in L^2(Q_1)$ and $F \in L^2(Q_1^*)$, there exists a $u \in C(-1,1;V(B_1^*))$ unique weak solution to 
		\begin{align*}
			\begin{cases}
			y^a \partial_t u - \operatorname{div}(y^a B(x) \nabla u)= \operatorname{div}(y^a F)&\hbox{in}~Q^*_1 \\
		y^a u_y\big|_{y=0}=f&\hbox{on}~Q_1\\
		u=0&\hbox{on}~\partial_{p}Q^*_1\setminus Q_1,
	\end{cases}
		\end{align*}
	where $V(B_1^*)=\left\{v\in H_a^1(B_1^*):\partial B_1^*\setminus\{y=0\}\right\}.$	\end{thrm}
\begin{proof}
	Consider $f \in L^2(Q_1)$ and $F \in L^2_a(Q_1^*).$ Now, using the density of smooth function, we have $f_k \in C^{\infty}_c(Q_1)$ and $F_k \in C^{\infty}_c(Q_1^*)$ such that $f_k \rightarrow f$ in $L^2(Q_1)$ and $F_k\rightarrow F$ in $L^2_a(Q_1^*).$ Define $g_k:[-1,1]\rightarrow V^*$ as following: Fix t, for $v \in V,$
	\begin{align*}
		g_k(t)v=\int_{B_1}f_k(t,x)v(x)dx + \int_{B_1^*}y^aF_k(t,X)\cdot \nabla v(X)dX.
	\end{align*}
Using triangle inequality, Cauchy-Schwarz inequality and trace theorem (in first term), we have $g_k(t) \in V^*$ with
\begin{align*}
||{g_k}(t)||_{V^*} \leq \bigg(\int_{B_1}|f_k(t,x)|^2dx\bigg)^{1/2}+\bigg(\int_{B_1^*}y^a|F_k(t,X)|^2dX\bigg)^{1/2}.
\end{align*} 
Use $f_k \in L^2(Q_1)$ and  $F_k \in L^2_a(Q_1^*)$ to get $g_k \in L^2(-1,1;V^*).$ Also, Define   $\tilde{g_k}:[-1,1]\rightarrow V^*$ as following: Fix t, for $v \in V,$
\begin{align*}
	\tilde{g_k}(t)v=\int_{B_1}(f_k)_t (t,x)v(x)dx +\int_{B_1^*}y^a(F_k)_t(t,X)v(X)dX,
\end{align*}
then $g_k'=\tilde{g_k} \in L^2(-1,1;V^*).$  For $u,v \in V,$ define
\begin{align*}
	a(u,v)=\int_{B^*_{1}}y^a\nabla u \cdot \nabla v  dX.
\end{align*}
Note that by Poincare inequality, for all $u \in V$ we have
\begin{align*}
	a(u,u) \geq C||u||^2.
\end{align*}
Thus we have verified conditions of Theorem 5.1 of \cite{DL} for $u_0=0$ and $\chi=0$. Hence we get $u_k \in L^2(-1,1;V)$ and  $(u_k)_t \in L^2(-1,1;V)$ such that for almost all $t$ and for all $v \in V$ we have 
\begin{align}\label{get}
	\int_{B^*_1}(u_k)_t v dX+\int_{B^*_{1}}y^a\nabla u_k \cdot \nabla v  dX=\int_{B_1}f_k(t,x)v(x)dx + \int_{B_1^*}y^aF_k(t,X)\cdot \nabla v(X)dX.
\end{align}
Put $v=u_k(t)$ in (\ref{get}) and integrate with respect to $t$ to get
\begin{align*}
	\int_{-1}^{1}	\int_{B^*_1}(u_k)_tu_k dXdt+\int_{-1}^{1}\int_{B^*_{1}}y^a\nabla u_k \cdot \nabla u_k  dXdt=\int_{-1}^{1}\int_{B_1}f_ku_kdxdt + \int_{-1}^{1}\int_{B_1^*}y^aF_k\cdot \nabla u_kdXdt.
\end{align*}
Since $t \rightarrow ||u_k(t)||_{V}$ is a absolutely continuous function, we get
\begin{align*}
	\frac{1}{2}\int_{B^*_1}u_k^2(1) dXdt+\int_{-1}^{1}\int_{B^*_{1}}y^a\nabla u_k \cdot \nabla u_k  dXdt=\int_{-1}^{1}\int_{B_1}f_ku_kdxdt + \int_{-1}^{1}\int_{B_1^*}y^aF_k\cdot \nabla u_kdXdt.
\end{align*}
On applying Holder's inequality, trace theorem and AM-GM inequality in first term of RHS and Holder's inequality and AM-GM inequality in second term we get 
\begin{align}\label{bound}
	\int_{-1}^{1}\int_{B^*_{1}}y^a\nabla u_k \cdot \nabla u_k  dXdt \leq 2C_{Tr}\int_{-1}^{1}\int_{B_{1}}f_k^2dxdt+2\int_{-1}^{1}\int_{B^*_{1}}y^a|F_k|^2dXdt.
\end{align}
Now write (\ref{get}) as
\begin{align}
	\int_{B^*_1}(u_k)_t v dX=\int_{B_1}f_k(t,x)v(x)dx + \int_{B_1^*}y^aF_k(t,X)\cdot \nabla v(X)dX-\int_{B^*_{1}}y^a\nabla u_k \cdot \nabla v  dX.
\end{align}
After applying Holder's inequality and trace theorem in first term of RHS and Holder's inequality in second term and third term of RHS we get, for almost all $t$
\begin{align*}
	||(u_k)_t(t)||_{V^*} \leq \bigg(\int_{B_1}|f_k(t,x)|^2dx\bigg)^{1/2}+\bigg(\int_{B_1^*}y^a|F_k(t,X)|^2dX\bigg)^{1/2}+\bigg(\int_{B_1^*}y^a|\nabla u_k|^2dX\bigg)^{1/2}.
\end{align*}
After Squaring both side and applying AM-GM inequality we get
\begin{align*}
	\int_{-1}^{1}||(u_k)_t(t)||_{V^*}^2dt \leq 4\int_{-1}^{1}\int_{B_1}|f_k(t,x)|^2dxdt+4\int_{-1}^{1}\int_{B_1^*}y^a|F_k(t,X)|^2dXdt+4\int_{-1}^{1}\int_{B_1^*}y^a|\nabla u_k|^2dXdt.
\end{align*}
Using (\ref{bound}) and boundedness of $||f_k||_{L^2}$ and $||F_k||_{L_a^2}$, we get $u_k$ and $(u_k)_t$ is bounded in $L^2(-1,1;V)$ and $L^2(-1,1;V^*)$ respectively. Hence we will get a weak convergent subsequence $u_k$ such that $u_k \rightharpoonup u$ in $L^2(-1,1;V)$ and $(u_k)_t \rightharpoonup u_t$ in $L^2(-1,1;V^*).$ Hence, for all $\phi \in H^1_0(-1,1;V)$ we have
\begin{align}\label{fin}
-\int_{-1}^{1}\int_{B^*_1}u \phi_t dXdt+\int_{-1}^{1}\int_{B^*_{1}}y^a\nabla u \cdot \nabla \phi  dXdt=\int_{-1}^{1}\int_{B_1}f(t,x)\phi dxdt + \int_{-1}^{1}\int_{B_1^*}y^aF_k(t,X)\cdot \nabla \phi(t,X)dXdt.
\end{align}
Also $u \in L^2(-1,1;V)$ and $u_t  \in L^2(-1,1;V^*)$ implies $u \in C(-1,1;V).$ Since $u_k(0)=0$ for all $k$, $u(0)=0.$ Note that (\ref{fin}) is equivalent to following: for all $-1 < t_1 <t_2 <1$ and $\phi \in H^1(-1,1;V)$  we have 
\begin{align}\label{fin2}
	\int_{B^*_1}u \phi(t_1) dX-\int_{B^*_1}u \phi(t_2) dXdt-\int_{t_1}^{t_2}\int_{B^*_1}u \phi_t dXdt+\int_{t_1}^{t_2}\int_{B^*_{1}}y^a\nabla u \cdot \nabla \phi  dXdt\\=\int_{t_1}^{t_2}\int_{B_1}f(t,x)\phi dxdt + \int_{t_1}^{t_2}\int_{B_1^*}y^aF(t,X)\cdot \nabla \phi(t,X)dXdt.
\end{align} 

Now we will prove uniqueness. Assume there are two solutions $u_1$ and $u_2$ in $C(-1,1;V) \cap L^2(-1,1;V)$. By standard argument Steklov average of $u^i$, $u^i_h$ belongs to $H^1(-1,1;V)$ and satisfy
\begin{align*}
	\int_{t_1}^{t_2}\int_{B^*_1}(u^i_h)_t \phi dXdt+\int_{t_1}^{t_2}\int_{B^*_{1}}y^a\nabla u^i_h \cdot \nabla \phi  dXdt=\int_{t_1}^{t_2}\int_{B_1}f_h(t,x)\phi dxdt + \int_{t_1}^{t_2}\int_{B_1^*}y^aF_h(t,X)\cdot \nabla \phi(t,X)dXdt.
\end{align*}
On Putting $\phi =u^2_h-u^1_h$ and subtracting two equations we get
\begin{align*}
	\int_{t_1}^{t_2}\int_{B^*_1}(u^2_h-u^1_h)_t (u^2_h-u^1_h) dXdt+\int_{t_1}^{t_2}\int_{B^*_{1}}y^a\nabla (u^2_h-u^1_h) \cdot \nabla (u^2_h-u^1_h) dXdt=0.
\end{align*}
On passing limit as $h \rightarrow 0$ we get
\begin{align*}
	\int_{B^*_1}(u^2-u^1)^2(t_2)dX-\int_{B^*_1}(u^2-u^1)^2(t_1)dX\leq 0
\end{align*} 
Take $t_1 \rightarrow -1$ and conclude $u^2=u^1.$ This completes the proof of uniqueness.
\end{proof}	
\begin{rmrk}\label{intexsit} 
	For every $F \in L^2(Q_1)$, there exists a unique weak solution $U \in C(-1,1;V (B_1))$ to
	\begin{align*}
		\begin{cases}
		 \partial_t u - \operatorname{div}( B(x) \nabla u)= \operatorname{div}( F)&\hbox{in}~Q_1 \\
		u=0&\hbox{on}~\partial_pQ_1.
	\end{cases}
	\end{align*}
This can be obtained by taking $V=H^1_{0}(B_1)$ and measure $dxdt$ instead of $y^adXdt.$ 
\end{rmrk}

\end{document}